\documentclass[12pt]{amsproc}

\usepackage[top=1in,bottom=1in,left=1.5in,right=1.5in]{geometry}
\usepackage{amsfonts} 
\usepackage{amsmath}
\usepackage{amssymb}
\usepackage{amsthm}
\usepackage{mathrsfs}
\usepackage{setspace}
\usepackage{subfigure}
\usepackage{url}
\usepackage{booktabs}
\usepackage{ifthen}
\usepackage{tikz}
\usetikzlibrary{calc}
\usepackage{enumerate}

\newcommand{\A}{\mathcal{A}}

\newcommand{\C}{\mathscr{C}}

\newcommand{\Sf}{\mathcal{S}}

\newtheorem{thm}{Theorem}[section]

\newtheorem{prop}[thm]{Proposition}
\newtheorem{cor}[thm]{Corollary}
\newtheorem{obs}[thm]{Observation}
\newtheorem{claim}[thm]{Claim}

\newtheorem{defn}[thm]{Definition}

\numberwithin{equation}{section}

\begin{document}

\title[Vizing's conjecture]{Vizing-type bounds for graphs with induced subgraph restrictions}

\author{Elliot Krop}
\address{Elliot Krop (\tt elliotkrop@clayton.edu)}
\author{Pritul Patel}
\address{Pritul Patel (\tt pritulpatel@clayton.edu)}
\address{Department of Mathematics, Clayton State University}
\author{Gaspar Porta}
\address{Gaspar Porta (\tt gaspar.porta@washburn.edu)}
\address{Department of Mathematics and Statistics, Washburn University}
\date{\today}

\begin {abstract}
For any graphs $G$ and $H$, we say that a bound is of Vizing-type if $\gamma(G\square H)\geq c \gamma(G)\gamma(H)$ for some constant $c$. We show several bounds of Vizing-type for graphs $G$ with forbidden induced subgraphs. In particular, if $G$ is a triangle and $K_{1,r}$-free graph, then for any graph $H$, $\gamma(G\square H)\geq \frac{r}{2r-1}\gamma(G)\gamma(H)$. If $G$ is a $K_r$ and $P_5$-free graph for some integer $r\geq 2$, then for any graph $H$, $\gamma(G\square H)\geq \frac{r-1}{2r-3}\gamma(G)\gamma(H)$. We do this by bounding the power of $G$, $\pi(G)$. We show that if $G$ is claw-free and $P_6$-free or $K_4$ and $P_5$-free, then for any graph $H$, $\gamma(G\square H)\geq \gamma(G)\gamma(H)$. Furthermore, we show Vizing-type bounds in terms of the diameter of $G$.
\\[\baselineskip] 2010 Mathematics Subject
      Classification: 05C69
\\[\baselineskip]
\end {abstract}

\maketitle

 \section{Introduction}
Vizing's conjecture \cite{Vizing}, now open for fifty-four years, states that for any two graphs $G$ and $H$,
\begin{align}
\gamma(G \square H) \geq \gamma(G)\gamma(H)\label{V}
\end{align}
where $\gamma(G)$ is the domination number of $G$.

The survey \cite{BDGHHKR} discusses many results and approaches to the problem. For more recent partial results see \cite{ST}, \cite{PPS}, \cite{B}, \cite{CK}, \cite{K}, and \cite{K2}.

A predominant approach to the conjecture has been to show it true for some large class of graphs. For example, in their seminal result, Bartsalkin and German \cite{BG} showed the conjecture for decomposable graphs. More recently, Aharoni and Szab\'{o} \cite{AS} showed the conjecture for chordal graphs and Bre\v{s}ar \cite{B} gave a new proof of the conjecture for graphs $G$ with domination number $3$. 

We say that a bound is of Vizing-type if $\gamma(G\square H)\geq c \gamma(G)\gamma(H)$ for some constant $c$, which may depend on $G$ or $H$. It is known \cite{ST} that all graphs satisfy the Vizing-type bound, 
\begin{align}
\gamma(G \square H) \geq \frac{1}{2}\gamma(G)\gamma(H)+\frac{1}{2}\min\{\gamma(G),\gamma(H)\}.\label{suentarr}
\end{align}

Restricting the graphs, but as a generalization of Bartsalkin and German's class of decomposable graphs, Contractor and Krop \cite{CK} showed
\begin{align*}
\gamma(G\square H)\geq \left(\gamma(G)-\sqrt{\gamma(G)}\right)\gamma(H)
\end{align*}

where $G$ belongs to $\A_1$, the class of graphs which are spanning subgraphs of domination critical graphs $G'$, so that $G$ and $G'$ have the same domination number and the clique partition number of $G'$ is one more than its domination number.

\medskip

To describe another Vizing-type bound \cite{K2} define the \emph{power} of a graph $\pi(G)$ as follows:

\begin{defn}
For a fixed $\gamma$-set $D$ of $G$, the \emph{allegiance} of $D$ with respect to $G$, $a_G(D)=\max_{v\in V(G)}\{|D\cap N[v]\}$.
\end{defn}

\begin{defn}
The \emph{power} of a graph $G$, $\pi(G)=\min_{D}\{a_G(D)\}$ taken over all $\gamma$-sets $D$ of $G$.
\end{defn}

The author then showed the Vizing-type bound for any graphs $G$ and $H$, 
\begin{align}
\gamma(G\square H)\geq \frac{\pi(G)}{2\pi(G)-1}\gamma(G)\gamma(H).\label{bound}
\end{align}

By the above inequality, one can produce improved Vizing-type bounds on classes of graphs by finding the maximum power of those classes.

In this paper we consider Vizing-type bounds for classes of graphs which do not contain one or more induced subgraphs. Some of our arguments are simple or direct applications of previous results such as formula \eqref{bound}, while others require more work.

We show that if $G$ is a triangle and $K_{1,r}$-free graph, then $\pi(G)\leq r$ which by \eqref{bound} implies that for any graph $H$, $\gamma(G\square H)\geq \frac{r}{2r-1}\gamma(G)\gamma(H)$. If $G$ is a $K_r$ and $P_5$-free graph for some integer $r\geq 2$, then $\pi(G)\leq r-1$ similarly implying that for any graph $H$, $\gamma(G\square H)\geq \frac{r-1}{2r-3}\gamma(G)\gamma(H)$. We show that if $G$ is $K_4$ and $P_5$-free or if $G$ is claw-free and $P_6$-free, then Vizing's conjecture holds for $G$. Furthermore, we show Vizing-type bounds in terms of the diameter of $G$.

\subsection{Basic notation}

All graphs $G(V,E)$ are finite, simple, connected, and undirected with vertex set $V$ and edge set $E$. We may refer to the vertex set and edge set of $G$ as $V(G)$ and $E(G)$, respectively.  For more on basic graph theoretic notation and definitions we refer to Diestel~\cite{Diest}. 
 
For any graph $H$, we say a graph $G$ is \emph{H-free} if $G$ contains no induced subgraphs isomorphic to $H$. A \emph{claw} is the graph $K_{1,3}$.
 
For any graph $G=(V,E)$, a subset $S\subseteq V$ \emph{dominates} $G$ if $N[S]=G$. The minimum cardinality of $S \subseteq V$, so that $S$ dominates $G$ is called the \emph{domination number} of $G$ and is denoted $\gamma(G)$. We call a dominating set that realizes the domination number a $\gamma$-set.

An \emph{independent dominating set} of a graph $G$ is a set of independent (pairwise mutually non-adjacent) vertices which dominate $G$. The size of a smallest independent dominating set of $G$ is denoted by $i(G)$.

 The \emph{Cartesian product} of two graphs $G_1(V_1,E_1)$ and $G_2(V_2,E_2)$, denoted by $G_1 \square G_2$, is a graph with vertex set $V_1 \times V_2$ and edge set $E(G_1 \square G_2) = \{((u_1,v_1),(u_2,v_2)) : v_1=v_2 \mbox{ and } (u_1,u_2) \in E_1, \mbox{ or } u_1 = u_2 \mbox{ and } (v_1,v_2) \in E_2\}$.

\section{Observations for graphs with forbidden induced subgraphs}

We will utilize the following result.

\begin{thm}[Bacs\'o and Tuza \cite{BT}]\label{BacsoTuza}
If a connected graph $G$ is $P_5$-free, then $G$ has a dominating set that induces a clique or $P_3$.
\end{thm}

\begin{prop}\label{triangle and star}
If $G$ is triangle and $K_{1,r}$-free, for any integer $r>1$, then $\pi(G)\leq r$.
\end{prop}

\begin{proof}
If $D$ is a minimum dominating set of $G$, and $u$ is any vertex in $V(G)$, then notice that since there are no triangles, $u$ can only be adjacent to independent vertices. Furthermore, since $G$ is $K_{1,r}$-free, $u$ can be adjacent to no more than $r-1$ independent vertices of $D$. However, if $u\in D$ and $u$ is adjacent to $r-1$ other vertices in $D$, then $a_G(D)=r$ and $\pi(G)\leq r$. 

\end{proof}

The following argument is due to Douglas Rall.

\begin{prop}\label{cograph}
If $G$ is $P_4$-free, also known as a cograph, then for any graph $H$, $\gamma(G\square H)\geq \gamma(G)\gamma(H)$.
\end{prop}
\begin{proof}
Any cograph may be constructed from $K_1$ by a sequence of disjoint union and join operations \cite{Br}. Since $G$ is connected, the last operation in its constuction must have been a join, which implies that $\gamma(G)$ is either $1$ or $2$. In either case, Vizing's conjecture holds \cite{BDGHHKR}.
\end{proof}

In the following proposition, the first argument is due to Douglas Rall.

\begin{prop}\label{clique and P5}\leavevmode
\begin{enumerate}
\item If $G$ is $K_4$ and $P_5$-free, then $\gamma(G\square H)\geq \gamma(G)\gamma(H)$.
\item If $G$ is $K_r$ and $P_5$-free, for any integer $r>4$, then $\pi(G)\leq r-1$.
\end{enumerate}
\end{prop}

\begin{proof}
Suppose $G$ is $K_4$ and $P_5$-free. By Theorem \ref{BacsoTuza}, $G$ has a dominating set that induces a clique or $P_3$.  Since $G$ is $K_4$-free, any dominating clique would have order at most $3$.  Therefore, $G$ must have domination number at most 3.  Since Vizing's conjecture is known for graphs with domination number one, two \cite{BDGHHKR}, and three \cite{B}, it follows that a graph $G$ that is $K_3$-free and $P_5$-free satisfies Vizing’s Conjecture.

If $G$ is $K_r$ and $P_5$-free for $r>4$, then again by Theorem \ref{BacsoTuza}, $G$ has a minimum dominating set which either induces $P_3$ or a clique. In the first case, $\gamma(G)=3$ and hence $G$ satisfies Vizing's conjecture \cite{B}. If $G$ has a minimum dominating set of size less than $r-1$, then the $\pi(G)< r-1$. Thus, we may assume that $G$ has a minimum dominating set $\Gamma$ which is a clique of size $r-1$. Notice that since $G$ is $K_r$-free, any vertex in $\Gamma$ has $r-2$ neighbors in $\Gamma$ and any vertex not in $\Gamma$ has at most $r-2$ neighbors in $\Gamma$.
\end{proof}

\begin{cor}\leavevmode
For any integer $r>1$, 
\begin{enumerate}
\item If $G$ is triangle and $K_{1,r}$-free, then for any graph $H$, $\gamma(G\square H)\geq \frac{r}{2r-1}\gamma(G)\gamma(H)$.
\item If $G$ is $K_r$ and $P_5$-free, then for any graph $H$, $\gamma(G\square H)\geq \frac{r-1}{2r-3}\gamma(G)\gamma(H)$.
\end{enumerate}
\end{cor}

\begin{proof}
The proof is an immediate application of \eqref{bound} to Proposition \ref{triangle and star} and Proposition \ref{clique and P5}
\end{proof}

\section{Claw and $P_6$-free graphs}

If $\Gamma=\{v_1,\dots, v_k\}$ is a minimum dominating set of $G$, then for any $i\in [k]$, define the set of \emph{private neighbors} for $v_i$, $P_i=\big\{v\in V(G)-\Gamma: N(v)\cap \Gamma = \{v_i\}\big\}$. For $S\subseteq [k]$, $|S|\geq 2$, we define the \emph{shared neighbors} of $\{v_i:i\in S\}$ as $P_S=\big\{v\in V(G)-\Gamma: N(v)\cap \Gamma=\{v_i: i\in S\}\big\}$.

For any $S\subseteq [k]$, say $S=\{i_1,\dots, i_s\}$ where $s\geq 2$, we may write $P_S$ as $P_{\{i_1,\dots, i_s\}}$ or $P_{i_1,\dots, i_s}$ interchangeably.

The following useful notation was introduced in \cite{K}.

For $i\in [k]$, let $Q_i=\{v_i\} \cup P_i$. We call $\mathcal{Q}=\{Q_1,\dots, Q_k\}$ the \emph{cells} of $G$. For any $I\subseteq [k]$, we write $Q_I=\bigcup_{i\in I}Q_i$ and call $\C(\cup_{i\in I}Q_i)=\bigcup_{i\in I}Q_i\cup\bigcup_{S\subseteq I}P_{S}$ the \emph{chamber} of $Q_I$. We may write this as $\C_{I}$.

For a vertex $h\in V(H)$, the \emph{$G$-fiber} of $h$, $G^h$, is the subgraph of $G\square H$ induced by $\{(g,h):g\in V(G)\}$. 

For a minimum dominating set $D$ of $G\square H$, we define $D^h=D\cap G^h$. Likewise, for any set $S\subseteq [k]$, $P_S^h=P_S \times \{h\}$, and for $i\in [k]$, $Q_i^h=Q_i\times \{h\}$. By $v_i^h$ we mean the vertex $(v_i,h)$. For any $I^h\subseteq [k]$, where $I^h$ represents the indices of some cells in $G$-fiber $G^h$, we write $\C_{I^h}$ to mean the chamber of $Q^h_{I^h}$, that is, the set $\bigcup_{i\in I^h}Q_i\cup\bigcup_{S\subseteq I^h}P^h_{S}$.

Any vertex $v\in V(G)\times V(H)$ is \emph{vertically dominated} if $(\{v\}\times N_H[h])\cap D \neq \emptyset$, namely, there exists a vertex $v\in Q_i$ and a vertex $(v,h')\in D$, such that $hh'\in E(H)$. Vertices that are not vertically dominated are called \emph{vertically undominated}.
For $i\in [k]$ and $h\in V(H)$, we say that the cell $Q_i^h$ is \emph{vertically dominated} if $(Q_i\times N_H[h])\cap D\neq\emptyset$. A cell which is not vertically dominated is \emph{vertically undominated}. Note that all vertices of a vertically undominated cell $Q_i^h$ are dominated by vertices $(u,h)\in D$.

\medskip

The following classical result forms the basis for our argument.

\begin{thm}[Allan and Laskar \cite{AL}]\label{AL}
If $G$ is claw-free, then $i(G)=\gamma(G)$.
\end{thm}

The next fact follows from the definition of claw-free graphs.

\begin{obs}\label{private}
For any claw-free graph $G$ with minimum independent dominating set $\{v_1,\dots,v_k\}$, for any $S\subseteq [k]$ with $|S|\geq 3$, $|P_S|=0$.
\end{obs}

Our theorem is an adaptation of the argument \cite{K} for the Vizing-type bound for claw-free graphs $G$, $\gamma(G\square H)\geq \frac{2}{3}\gamma(G)\gamma(H)$.

\begin{thm}\label{claw-free}
If $G$ is a claw and $P_6$-free graph, then for any graph $H$, $\gamma(G \square H)\geq \gamma(G)\gamma(H)$.
\end{thm}

\begin{proof}
Let $G$ be a claw and $P_6$-free graph and $H$ any graph. We apply Theorem \ref{AL} and define a minimum independent dominating set of $G$, $\Gamma=\{v_1,\dots, v_k\}$. Let $D$ be a minimum dominating set of $G\square H$. 

We define a series of labelings of the vertices of $D$ so that projection onto $H$ of those vertices with labels containing a fixed element produces a dominating set of $H$. In all instances, for any $i,j\in [k]$ and $h\in V(H)$, if $v\in P^h_{i,j}$, then $v$ may be labeled by singleton labels $i,j,$ or paired labels $(i,j)$. 

Our goal is to reduce the number of paired labels as much as possible.

For any $h\in V(H)$, suppose the fiber $G^h$ contains $\ell_h(=\ell)$ vertically undominated cells $U=\big\{Q_{i_1}^h,\dots, Q_{i_{\ell}}^h\big\}$ for some $0\leq \ell \leq k$. We set $I^h=\{i_1,\dots, i_{\ell}\}$. 

\medskip

\noindent\emph{Labeling 1}:

If a vertex of $D^h$ for any $h\in H$, is in $Q_i^h$ for $1\leq i \leq k$, then we label that vertex by $i$. If $v\in D^h$ is a shared neighbor of some subset of $\{v_i:i\in I^h\}$, then by Observation \ref{private}, it is a member of $P^h_{i,j}$ for some $i,j\in I^h$, and we label $v$ by the pair of labels $(i,j)$. If $v$ is a member of $D\cap P^h_{i,j}$ for $i\in I^h$ and $j\in [k]-I^h$, then we label $v$ by $i$. If $v$ is a member of $D\cap P^h_{i,j}$ for $i,j \in [k]-I^h$, then we label $v$ by either $i$ or $j$ arbitrarily. 

\medskip

After Labeling $1$, all vertices of $D$ have a singleton label or a paired label. Next we relabel the vertices of $D$, doing so in $D^h$ for every fixed $h\in H$. 

\medskip

\noindent\emph{Labeling 2}:

For a fixed $h\in H$, suppose $v$ is some shared neighbor of a subset of $\{v_i: i\in I^h\}$ in the chamber of $Q_{I^h}^h$, which is vertically dominated, say by $y\in D^{h'}$ for some $h'\in H,\,h\neq h'$. More precisely, suppose $v\in P^h_{j_1,j_2}$ for some $j_1,j_2\in I^h$ which implies that $y\in P^{h'}_{j_1,j_2}$. 

The vertex $y$ may be labeled by a singleton or or paired label, regardless of whether Labeling $2$ had been performed on $D^{h'}$.

Suppose that $y$ is labeled by a singleton label, say $j_1$. If $D^h$ contains a vertex $x\in P_{j_1,j_2}^h$, then we remove the paired label $(j_1,j_2)$ from $x$ and relabel $x$ by $j_2$. 

Suppose $y$ is labeled by the paired label, $(j_1,j_2)$. If $D^h$ contains a vertex $x\in P_{j_1,j_2}^h$, then we remove the paired label $(j_1,j_2)$ from $x$ and then relabel $x$ arbitrarily by one of the singleton labels $j_1$ or $j_2$. This completes Labeling $2$.

\medskip

After Labeling $2$, a vertex $v$ of $D$ may have a paired label $(i,j)$ if $v\in P^h_{i,j}$ and for any $h'\in N_H(h)$, $D^{h'}\cap P^{h'}_{i,j}=\emptyset$.

\medskip

\noindent\emph{Labeling 3}:

For every $h\in H$, if $D^h$ contains vertices $x$ and $y$ both with paired labels $(j_1,j_2)$, for some integers $j_1,j_2,$, then we relabel $x$ by the label $j_1$ and $y$ by the label $j_2$. For every $h\in H$, if $D^h$ contains vertices $x$ and $y$ with paired labels $(j_1,j_2)$ and $(j_2,j_3)$ respectively, for some integers $j_1,j_2,$ and $j_3$, then we relabel $y$ by the label $j_3$. If $x$ and $y$ are labeled $j_1$ and $(j_1,j_2)$ respectively, for some integers $j_1,j_2$, we relabel $y$ by $j_2$. We apply this relabeling to pairs of vertices of $D^h$, sequentially, in any order.

\medskip

%
%

For $h\in H$, let $S_1^h$ be the vertices of $D^h$ which still have a pair of labels. Notice that after Labeling $3$, $S_1^h$ is contained in $\C_{I}$. For each vertex in $S_1^h$, we place each component of the paired label on that vertex in the set $J^h_1$. For example, if $S_1^h$ contains vertices with labels $(i_1,i_2)$ and $(i_3,i_4)$, then $J^h_1=\{i_1,i_2,i_3,i_4\}$.

Define the index set $I^h_1=[k] - I^h=\{i_{\ell+1}, \dots, i_{k}\}$ for vertically dominated cells of $G^h$. 





The following observations follow from the definition of claw-free:

\begin{enumerate}
\item For $j_1,j_2\in [k]-I^h$, no vertex of $D\cap P^h_{j_1,j_2}$ may dominate any of $v_{i_1}^h,\dots,v_{i_{\ell}}^h$. Thus, $\{v_{i_1}^h,\dots, v_{i_\ell}^h\}$ must be dominated horizontally in $G^h$ by shared neighbors of $\{v_i^h:i\in I^h\}$ from the chamber of $Q_{I^h}^h$.
\item If $j_1,j_2,j_3,j_4$ are distinct elements of $[k]$ and $x\in P^h_{j_1,j_2}, y\in P^h_{j_3,j_4}$, then $x$ is not adjacent to $y$.
\item Similarly, $x\in P^h_{j_1}$ is not adjacent to any $y\in P^h_{j_2,j_3}$.
\item By $(2)$, all vertices of $D^h-\C_{J_1^h}$ which are adjacent to some vertex of $\C_{J_1^h}$ must be members of $P^h_{i}$ for $i\in I^h_1$.
\item If a vertex of $\C_{J_1^h}$ is vertically undominated and dominated from outside $\C_{J_1^h}$, then it must be a member of $P^h_j$ for some $j\in J^h_1$, since neither shared neighbors of $\C_{J_1^h}$, nor $v_j^h$ for $j\in J_1^h$, can be adjacent to vertices outside $\C_{J_1^h}$.
\end{enumerate}

Observations $(1)-(5)$ imply the following:

\begin{claim}\label{horizdom}
If $v$ is a vertically undominated vertex of $\C_{J_1^h}$ which is not dominated by a shared neighbor in $\C_{J_1^h}$, then it is a private neighbor in $\C_{J_1^h}$. Furthermore, $v$ must be dominated by a private neighbor of $\C_{I_1^h}$.
\end{claim}


Suppose every vertex of $D$ is labeled by a single label. For any $i$, $1\leq i \leq k$, projecting all vertices labeled by $i$ onto $H$ produces a dominating set of $H$. Summing over all $i$ we count at least $\gamma(G)\gamma(H)$ vertices of $D$.

If for some $h\in V(H)$, some vertex $v\in D^h$ is labeled by a paired label, then $v\in \C_{J_1^h}$ and for some $i,j\in J_1^h$, $v\in P_{i,j}^h$. Since labelings $2$ and $3$ have been performed, $v$ must be the only vertex in $P_{i,j}^h$, else $i,j$ would not be in $J_1^h$. Furthermore, if $v$ cannot dominate both $Q_i^h$ and $Q_j^h$, since this would produce a dominating set of $G$ with size less that $\gamma(G)$. Thus, some private neighbor $p$ of $v_i^h$ or $v_j^h$ must be independent from $v$. Without loss of generality, suppose $p\in P_j$. By Claim \ref{horizdom}, $p$ must be horizontally dominated by some vertex $q$ which is a private neighbor in $\C_{I_1^h}$. That is, for some $m\in I_1^h$, $q\in P_m$. However, this produces $P_6: v_ivv_jpqv_m$, which is a contradiction.

\end{proof}

\section{Vizing-type inequalites in terms of diameter}

We now review some significant ideas from \cite{BR}, which generalize the seminal work of Bartsalkin and German \cite{BG}. 

\begin{defn}\label{FR}
For pairwise disjoint sets of vertices $S_1, \dots, S_k$ from a graph $G$, with $\Sf=S_1\cup\dots\cup S_k$ and $Z=V(G)-\Sf$, we say $S_1, \dots, S_k$ form a \emph{fair reception of size $k$} if the following condition holds:

For any integer $\ell$, $1\leq \ell \leq k$, and any choice of $\ell$ sets $S_{i_1}, \dots, S_{i_{\ell}}$, if $D$ externally dominates $S_{i_1}\cup \dots \cup S_{i_{\ell}}$, then
\[\left|D\cap Z\right| + \sum_{j,S_j\cap D\neq \emptyset}\left(\left|S_j \cap D\right|-1\right)\geq \ell\]
\end{defn}

For any graph $G$, the largest $k$ such that there exists a fair reception of size $k$ in $G$ is called the \emph{fair domination number of $G$} and is denoted by $\gamma_F(G)$.

\begin{thm}[Bre\v{s}ar and Rall \cite{BR}]\label{fair}
\[\gamma(G \square H) \geq \max \{\gamma(G)\gamma_F(H),\gamma_F(G)\gamma(H)\} \]
\end{thm}

\medskip

The distance between two vertices $u$ and $v$ is the number of edges in a shortest path between them. For any vertex $v$, the eccentricity of $v$, $\varepsilon(v)$, is the greatest distance from $v$ to any other vertex. The diameter of a graph $G$, $d(G)$, can now be defined as
\[d(G)=\max_{v\in V(G)}\varepsilon(v).\]

\subsection{Graphs with large diameter}

It is easy to see that graphs with large diameter admit a large fair reception.

\begin{prop}\label{diam}
$\gamma(G\square H)\geq \left(\lfloor\frac{d(G)}{3}\rfloor+1\right)\gamma(H)$
\end{prop}

\begin{proof}
Let $P$ be an induced path in $G$ of length $d=d(G)$ with end vertex $a$. Define the \emph{$i^{th}$ level set} of $G$, $V_i$, for $0\leq i\leq d$, as the set of vertices of $G$ of length $i$ from $a$. We produce a fair reception of the required size depending on the congruence class of $d$. 

\medskip

If $d\equiv 0 \pmod 3$, let $k=\frac{d}{3}+1$. We partition the vertices of $G$ by the level sets to form a fair reception. This is done so that one set of the fair reception contains the first and second level sets and another set of the fair reception contains the next to last and the last level set. The rest of the sets of the fair reception each contain three consecutive level sets. That is, define $S_1=V_0\cup V_1$ and $S_k=V_{k-1}\cup V_k$. For $j=2+3i$ and $0\leq i \leq \frac{d}{3}-2$, let $S_{i+2}=V_j\cup V_{j+1}\cup V_{j+2}$. We note that no set $S_i$ can be externally dominated from any other set $S_j$ for $i\neq j$ and $1\leq i,j\leq k$. Thus, $\Sf=\{S_1,\dots,S_{k}\}$ and $Z=\emptyset$ form a fair reception.

\medskip

If $d\equiv 1 \pmod 3$, let $k=\frac{d+2}{3}$. We partition the vertices of $G$ by level sets for a fair reception so that one set of the fair reception contains the first and second level sets and the rest of the sets of the fair reception each contain three consecutive level sets. That is, define $S_1=V_0\cup V_1$. For $j=2+3i$ and $0\leq i \leq \frac{d}{3}-1$, let $S_{i+2}=V_j\cup V_{j+1}\cup V_{j+2}$. We note that no set $S_i$ can be externally dominated from any other set $S_j$ for $i\neq j$ and $1\leq i,j\leq k$. Thus, $\Sf=\{S_1,\dots,S_{k}\}$ and $Z=\emptyset$ form a fair reception.

\medskip

If $d\equiv 2 \pmod 3$, let $k=\frac{d+1}{3}$. We partition the vertices of $G$ by level sets for a fair reception so that the sets of the fair reception each contain three consecutive level sets. That is, for $j=3i-3$ and $1\leq i \leq k$, let $S_i=V_j\cup V_{j+1}\cup V_{j+2}$. We note that no set $S_i$ can be externally dominated from any other set $S_j$ for $i\neq j$ and $1\leq i,j\leq k$. Thus, $\Sf=\{S_1,\dots,S_{k}\}$ and $Z=\emptyset$ form a fair reception.

\medskip

Notice that $k=\lfloor\frac{d(G)}{3}\rfloor+1$ and $\gamma_F(G)\geq k$. Hence, by Theorem \ref{fair}, we produce the proposed inequality.
\end{proof}

We note that this bound is an improvement over \eqref{suentarr} for graphs $G$ such that $d(G)>\frac{3}{2}\gamma(G)$.

\section{Discussion}

We would like to note that apart from the study of Vizing-type inequalities for graphs with forbidden subgraphs, the results of this paper can be viewed, in part, as an investigation of Vizing-type bounds for graphs with different fixed diameters. Proposition \ref{cograph} shows us that Vizing's inequality holds for graphs with diameter $2$. For graphs with higher diameter, we could only make approximate statements. For diameter $3$ we have Proposition \ref{clique and P5} which is a Vizing-type statement that further relies on the exclusion of certain cliques. For diameter $4$, Theorem \ref{claw-free} guarantees Vizing's bound but only for claw-free graphs. Thus, even for graphs of small diameters, Vizing's conjecture is far from resolved. As the diameter gets large, Proposition \ref{diam} starts becoming more relevant, but we have no bounds exceeding \eqref{suentarr} until $d(G)>\frac{3}{2}\gamma(G)$.

\section{Acknowledgements}
We would like to thank Douglas Rall for the helpful comments.

 \bibliographystyle{plain}
 
 \end{document}